\documentclass{article}
\usepackage[british]{babel}
\usepackage{amsthm}
\usepackage{mathtools}
\usepackage{amssymb, enumerate}
\usepackage{latexsym}
\usepackage{url}
\usepackage[all,cmtip]{xy}
\usepackage{amsmath,amssymb}
\usepackage{xcolor}

\newcounter{itemcounter}
\numberwithin{itemcounter}{section}

\newtheorem{thm}[itemcounter]{Theorem}
\newtheorem{lem}[itemcounter]{Lemma}

\newtheorem{prop}[itemcounter]{Proposition}

\newtheorem{exam}[itemcounter]{Example}
\newtheorem*{thm*}{Theorem}
\newtheorem*{con*}{Conjecture}
\newtheorem*{cor*}{Corollary}
\newtheorem*{ack*}{Acknowledgements}

\newcommand{\Inf}{\mathop{\rm Inf}\nolimits}

\newcommand{\Irr}{\mathop{\rm Irr}\nolimits}

\newcommand{\Hom}{\mathop{\rm Hom}\nolimits}

\newcommand{\Aut}{\mathop{\rm Aut}\nolimits}

\newcommand{\Out}{\mathop{\rm Out}\nolimits}
\newcommand{\Inn}{\mathop{\rm Inn}\nolimits}

\newcommand{\Pic}{\mathop{\rm Pic}\nolimits}
\newcommand{\rk}{\mathop{\rm rk}\nolimits}

\newcommand{\nth}{\mathop{\rm th}\nolimits}

\newcommand{\Id}{\mathop{\rm Id}\nolimits}

\newcommand{\Picent}{\mathop{\rm Picent}\nolimits}

\newcommand{\cC} {\mathcal{C}}
\newcommand{\cT} {\mathcal{T}}

\newcommand{\cO} {\mathcal{O}}

\newcommand{\cF} {\mathcal{F}}

\title{On Picent for blocks with normal defect group\footnote{This research was supported by the EPSRC (grant no EP/T004606/1).} }
\author{Michael Livesey\footnote{School of Mathematics, University of Manchester, Manchester, M13 9PL, United Kingdom. Email: michael.livesey@manchester.ac.uk} and Claudio Marchi\footnote{School of Mathematics, University of Manchester, Manchester, M13 9PL, United Kingdom. Email: claudio.marchi@manchester.ac.uk}}
\date{}

\begin{document}

\maketitle

\begin{abstract}
We prove that if $b$ is a block of a finite group with normal abelian defect group and inertial quotient a direct product of elementary abelian groups, then $\Picent(b)$ is trivial. We also provide examples of blocks $b$ of finite groups with non-trivial $\Picent(b)$. We even have examples with normal abelian defect group and abelian inertial quotient.
\end{abstract}

\section{Introduction}

Let $\cO$ be a complete discrete valuation ring with $k:=\cO/J(\cO)$ an algebraically closed field of prime characteristic $p$. Let $K$, the field of fractions of $\cO$, have characteristic zero. We take $K$ large enough, meaning that it contains all $|H|^{\nth}$ roots of unity, for all the groups $H$ involved in the rest of the paper. By a block we always mean a block $b$ of $\cO H$ for some finite group $H$.  We use $\Irr(H)$ to denote the set of irreducible characters of $H$ and $\Irr(b)$ the subset of irreducible characters lying in the block $b$.
\newline
\newline
Let $b$ be a block of a finite group $H$. The Picard group $\Pic(b)$ of $b$ consists of isomorphism classes of $b$-$b$-bimodules which induce $\cO$-linear Morita auto-equivalences of $b$. For $M,N\in\Pic(b)$, the group multiplication is given by $M\otimes_{b}N$. We will often view $M$ as an $\cO(H\times H)$-module via $(g,h).m=gmh^{-1}$, for all $g,h\in H$ and $m\in M$. This paper is concerned with $\Picent(b)$, the subgroup of $\Pic(b)$ consisting of Morita auto-equivalences that induce the trivial permutation of $\Irr(b)$.
\newline
\newline
In~\cite{eali19} it is proved that all $2$-blocks with abelian defect group of rank at most three have trivial $\Picent$. Our main result is that in certain situations $\Picent$ is always trivial (see Theorems~\ref{thm:elem_ab} and~\ref{thm:cyc_prin}).

\begin{thm*}
Let $b$ be a block with normal abelian defect group and cyclic inertial quotient or inertial quotient a product of elementary abelian groups or a principal block with abelian inertial quotient. Then $\Picent(b)$ is trivial.
\end{thm*}
Notwithstanding the above theorem, in this paper we show that blocks with non-trivial $\Picent$ do exist. Our focus is entirely on blocks with normal defect group. We construct three families of examples. The first, see Example~\ref{exam:OP}, is simply any $p$-group with a non-inner, class-preserving automorphism. The second example, see Proposition~\ref{prop:st_example}, is given by a non-inner, class-preserving automorphism of a group $H$, with normal abelian Sylow $p$-subgroup, that induces a non-trivial element of $\Picent$ for the principal block of $\cO H$. Our final example, see Proposition~\ref{prop:tensor_linear}, is concerned with a non-principal block with normal abelian defect group and abelian inertial quotient. An interesting point to note in this final case is that the relevant bimodule has vertex $\Delta D$, where $D$ is the defect group. This means the Morita auto-equivalence is simply given by tensoring with a linear character of the inertial quotient.
\newline
\newline
The following notation will hold for the remainder of the article. If $N\lhd H$, for a finite group $H$, and $\chi\in\Irr(N)$, then we denote by $\Irr(H|\chi)$ the set of irreducible characters of $H$ appearing as constituents of $\chi\uparrow_N^H$. Similarly, for a block $b$ of $H$, we define $\Irr(b|\chi):=\Irr(b)\cap\Irr(H|\chi)$. If $h\in H$, then we denote by $c_h\in\Aut(H)$ the automorphism given by $g\mapsto {}^hg:=hgh^{-1}$. For any $\cO N$-module $M$, ${}^h M$ will denote the $\cO N$-module equal to $M$ as an $\cO$-module but with the action of $N$ defined via $g.m=c_h^{-1}(g)m$, for all $g\in N$ and $m\in {}^h M$. The character ${}^h\chi\in\Irr(N)$ is defined analogously. If $X\leq H$, then we set $I_X(M):=\{h\in X|{}^hM\cong M\}$ and define $I_X(\chi)$ analogously. We write $\cO_H$ for the $\cO H$-module $\cO$ with the trivial action of $H$ and $1_H\in\Irr(H)$ will denote the trivial character of $H$. We use $e_b\in\cO H$ to signify the block idempotent of $b$ and $e_\eta\in KH$ the character idempotent associated to $\eta\in\Irr(H)$.
\newline
\newline
If $\psi\in\Aut(H)$, then we denote by $\Delta\psi:=\{(h,\psi(h))|h\in H\}\leq H\times H$. If $\psi=\Id_H$, then we just denote $\Delta\psi$ by $\Delta H$. We will also view $\psi$ as an $\cO$-linear automorphism of $\cO H$, where appropriate. For an $\cO$-algebra $A$, $\Aut_{\cO}(A)$ will stand for the set of $\cO$-algebra automorphisms of $A$. If $\alpha\in\Aut_{\cO}(b)$, we denote by ${}_\alpha b\in\Pic(b)$ the equivalence induced by $\alpha$. In other words, ${}_\alpha b=b$ as sets but with $x.m.y:=\alpha(x)my$, for all $x,y\in b$ and $m\in{}_\alpha b$.
\newline
\newline
The article is organised as follows. $\S2$ contains an assortment of lemmas that will be needed in $\S3$ and $\S4$. $\S3$ concerns Morita auto-equivalences induced by bimodules with trivial source and then goes on to prove our main theorems regarding families of blocks with trivial $\Picent$. $\S4$ gives our three classes of examples of blocks with non-trivial $\Picent$.

\section{Preliminaries}

In this section we gather together various lemmas that will be used throughout the rest of the article.

\begin{lem}\label{lem:vec_spc}
Let $V$ be an $n$-dimensional vector space over a field $\mathbb{F}$, for some $n\in\mathbb{N}$. In addition let $(V_i)_{i\in I}$ be a set of subspaces of $V$ of codimension one such that $\bigcap_{i\in I}V_i=\{0\}$. Then there exists a subset $\{i_1,\dots,i_n\}\subseteq I$ such that $\bigcap_{j=1}^nV_{i_j}=\{0\}$. Furthermore, setting $U_l:=\bigcap_{\substack{j=1\\j\neq l}}^nV_{i_j}$, for each $1\leq l\leq n$, we have that each $U_l$ has dimension one and $V=\bigoplus_{l=1}^nU_l$.
\end{lem}

\begin{proof}
Choose any $i_1\in I$. Then choose the remaining $i_l$'s iteratively by demanding that $\bigcap_{j=1}^lV_{i_j}$ is strictly contained in $\bigcap_{j=1}^{l-1}V_{i_j}$ for $2\leq l\leq n$. Note that the next $V_{i_l}$ always exists, until $\bigcap_{j=1}^lV_{i_j}=\{0\}$, since $\bigcap_{i\in I}V_i=\{0\}$. In fact $\bigcap_{j=1}^lV_{i_j}=\{0\}$ precisely when $l=n$, since $\bigcap_{j=1}^{l-1}V_{i_j}$ has codimension one in $\bigcap_{j=1}^lV_{i_j}$, for $1\leq l\leq n$. In particular, each $U_l$ has dimension one. Therefore, the final claim follows since
\begin{align*}
\left(\sum_{\substack{l=1\\l\neq m}}^nU_l\right)\cap U_m\subseteq V_{i_m}\cap U_m=\bigcap_{j=1}^nV_{i_j}=\{0\},
\end{align*}
for all $1\leq m\leq n$.
\end{proof}

\begin{lem}\label{lem:exten}
Let $H$ be a finite group, $N$ a normal subgroup such that $H/N$ is abelian, $\ell$ a prime and $N\leq L\leq H$ such that $L/N$ is a cyclic $\ell$-group and $\ell\nmid[H:L]$. If $\chi\in\Irr(N)$ is $H$-stable, then $\chi$ extends to $L$ and every extension is $H$-stable.
\end{lem}

\begin{proof}
Since $L/N$ is cyclic and $\chi$ is $L$-stable, $\chi$ certainly extends to $L$. Now let $h\in H$. Since $H/N$ is abelian, $L/N$ is a cyclic $\ell$-group and $\ell\nmid[H:L]$, $L\langle h\rangle/N$ is cyclic. Therefore, $\chi$ extends to $L\langle h\rangle$ in $[L\langle h\rangle:N]$ different ways. In particular, every extension of $\chi$ to $L$ extends to $L\langle h\rangle$ and so $h$ stabilises $\chi$.
\end{proof}

For a finite $p$-group $P$, we denote by $\Phi(P)$ the Frattini subgroup of $P$. For the following lemma we have in mind the semi-direct product $P\rtimes H$. We will borrow notation from this setup, for example $C_P(H)$ will denote the set of fixed points in $P$ under the action of $H$.

\begin{lem}\label{lem:HonP}
Let $P$ be a finite abelian $p$-group and $H$ an abelian $p'$-subgroup of $\Aut(P)$.
\begin{enumerate}
\item $P=C_P(H)\times[H,P]$.
\item The natural homomorphism $H\to\Aut(P/\Phi(P))$ is injective. Furthermore, $H$ acts indecomposably on $P$, that is there does not exist a non-trivial, $H$-invariant decomposition $P=P_1\times P_2$, if and only if $H$ acts indecomposably on $P/\Phi(P)$.
\item The homomorphism $H\to\Aut(\Irr(P))$, $h(\lambda)(x)=\lambda(h^{-1}(x))$, for all $h\in H$, $x\in P$ and $\lambda\in\Irr(P)$, is injective and this action of $H$ on $\Irr(P)$ is indecomposable if and only if the action of $H$ on $P$ is.
\item If $H$ does act indecomposably on $P$, then $H$ is cyclic and $C_P(h)$ is trivial for all $h\in H\backslash\{1\}$.
\item Let $\psi\in C_{\Aut(P)}(H)$ such that, for all $x\in P$, $\psi(x)=h(x)$, for some $h\in H$. Then $\psi\in H$.
\end{enumerate}
\end{lem}

\begin{proof}$ $
\begin{enumerate}
\item This is~\cite[$\S5$, Theorem 2.3]{gor80}.
\item The fact that the homomorphism is injective follows from~\cite[$\S5$, Theorem 1.4]{gor80}. The indecomposablity statement is~\cite[Lemma 2.5]{li2019}.
\item We can identify $P$ with $\Irr(\Irr(P))$ as $H$-sets via $x\mapsto(\lambda\mapsto\lambda(x))$. Injectivity follows since, if some $h\in H$ fixes $\Irr(P)$ pointwise, then $h$ also fixes $\Irr(\Irr(P))$ pointwise and hence also $P$. If $P=P_1\times P_2$ is an $H$-invariant decomposition of $P$, then $\Irr(P)=\Irr(P,1_{P_1})\times\Irr(P,1_{P_2})$ is an $H$-invariant decomposition of $\Irr(P)$. The reverse implication follows, once again, by identifying $P$ with $\Irr(\Irr(P))$.
\item \cite[Lemma 2.6]{li2019} gives that $H$ is cyclic. Next let $h\in H$. If $C_P(h)$ is non-trivial then, by part (1), $P=C_P(h)\times[\langle h\rangle,P]$ is an $H$-invariant decomposition. Therefore, since $H$ acts indecomposably on $P$, $C_P(h)=P$ and $h=1$.
\item Let's first assume that $P$ is elementary abelian and $H$ acts indecomposably. By post-composing with a suitable element of $H$, we may assume that $\psi$ has a non-trivial fixed point $x\in P$. Now $\psi$ must fix all of $P$, since otherwise $C_P(\psi)$ is a non-trivial, proper $H$-invariant subgroup of $P$ which, by~\cite[$\S3$, Theorem 3.2]{gor80}, contradicts the indecomposability of the action of $H$ on $P$.
\newline
\newline
We next drop the assumption that $H$ acts indecomposably. Decompose $P=P_1\times\dots\times P_n$ into indecomposable components. First note that the hypotheses of the lemma ensure that $\psi$ respects this decomposition. Now choose $x_i\in P_i\backslash\{1\}$, for $1\leq i\leq n$. As above, we may assume that $\psi$ fixes $(x_1,\dots,x_n)\in P$ and hence each $x_i$. Also as above, we must now have $\psi$ acting as the identity on each $P_i$ and therefore all of $P$.
\newline
\newline
For the general case, by the previous paragraph, we may assume that $\psi$ induces the trivial automorphism of $P/\Phi(P)$. Decompose $P=P_1\times\dots\times P_n$ into indecomposable components. Let $x_i\in P_i\backslash\Phi(P_i)$, for $1\leq i\leq n$ and set $x:=(x_1,\dots,x_n)\in P$. Then
\begin{align*}
C_H(x\Phi(P))=\bigcap_{i=1}^nC_H(x_i\Phi(P_i))=\bigcap_{i=1}^nC_H(P_i/\Phi(P_i))=\bigcap_{i=1}^nC_H(P_i)=\{1\},
\end{align*}
where the second equality follows from part (4) and the third follows from part (2). Therefore, any $h\in H$ such that $\psi(x)=h(x)$ must be trivial. So $\psi$ is trivial on $(P_1\backslash\Phi(P_1))\times\dots\times(P_n\backslash\Phi(P_n))$ and hence on all of $P$.
\end{enumerate}
\end{proof}

\section{Blocks with trivial $\Picent$}

We set the following notation that will hold for the remainder of this section. Let $D$ be a finite $p$-group, $E$ a finite $p'$-group and $Z\leq E$ a central, cyclic subgroup such that we can identify $L:=E/Z$ with a subgroup of $\Aut(D)$. Through this identification we define $G:=D\rtimes E$ and $B:=\cO Ge_\varphi$ for some fixed $\varphi\in\Irr(Z)$. Since $D\lhd G$, any block idempotent of $\cO G$ is supported on $C_G(D)=Z(D)\times Z$. Therefore, $B$ is a block of $\cO G$ with defect group $D$. We set $C:=D\times Z$. $\cT(B)$ will denote the subgroup of $\Pic(B)$ consisting of bimodules that have trivial source when viewed as $\cO(G\times G)$-modules.

\begin{prop}\label{prop:triv_source}
Let $M\in\cT(B)$.
\begin{enumerate}
\item $M$ has vertex $\Delta\psi$ for some $\psi\in N_{\Aut(D)}(L)$.
\end{enumerate}
We denote by $\psi_L\in\Aut(L)$ the automorphism of $L$ induced by conjugation by $\psi$.
\begin{enumerate}
\item[2.] $M\cong {}_\alpha B$, where $\alpha\in\Aut_{\cO}(B)$ satisfies $\alpha(xe_\varphi)=\psi(x)e_\varphi$, for all $x\in D$ and $\alpha(\cO gZe_\varphi)=\cO\psi_L(gZ)e_\varphi$, for all $g\in E$.
\item[3.] If $M$ has vertex $\Delta D$, then $M$ is just given by tensoring with some linear character $\lambda\in\Irr(G|1_C)$. Moreover, different $\lambda$'s give non-isomorphic $M$'s.
\end{enumerate}
\end{prop}

\begin{proof}$ $
\begin{enumerate}
\item By~\cite[Theorem 1.1(i)]{bkl18}, $M$ has vertex $\Delta\psi$ for some $\psi\in\Aut(D,\cF)$, where $\cF$ is the fusion system on $D$ determined by $B$ and $\Aut(D,\cF)$ is the subgroup of $\Aut(D)$ which stabilises $\cF$. In particular, $\psi\in N_{\Aut(D)}(L)$.
\item Since $M$ has vertex $\Delta \psi$ and $\cO_{\Delta \psi}\uparrow_{\Delta \psi}^{D\times D}\cong{}_\psi(\cO D)$ is indecomposable, $M$ is a direct summand of ${}_\psi(\cO D)\uparrow_{D\times D}^{G\times G}$. Now let $(g,h)\in E\times E$. Then ${}^{(g,h)}({}_{\psi}(\cO D))$ has vertex
\begin{align}
\begin{split}\label{algn:vert}
{}^{(g,h)}(\Delta \psi)=&\{({}^gx,{}^h\psi(x)|x\in D\}=\{(x,{}^h\psi({}^{g^{-1}}x))|x\in D\}\\
=&\Delta(hZ\circ \psi\circ g^{-1}Z),
\end{split}
\end{align}
where we are viewing $gZ,hZ\in\Aut(D)$. Now if $hZ=\psi_L(gZ)$, then
\begin{align*}
hZ\circ \psi\circ g^{-1}Z=\psi_L(gZ)\circ \psi\circ g^{-1}Z=\psi \circ gZ\circ\psi^{-1}\circ \psi\circ g^{-1}Z=\psi.
\end{align*}
Therefore, ${}^{(g,h)}(\Delta \psi)=\Delta \psi$ and $S_\psi\leq I_{G\times G}({}_\psi(\cO D))$, where $S_\psi\leq G\times G$ is generated by $D\times D$ and
\begin{align*}
E_\psi:=\{(g,h)\in E\times E|\psi_L(gZ)=hZ\}.
\end{align*}
If $I_{G\times G}({}_\psi(\cO D))$ properly contains $S_\psi$ then there exists some $(g,1)\in I_{G\times G}({}_\psi(\cO D))$, with $g\in E\backslash Z$. Since $p\nmid|L|$, $gZ$ has order prime to $p$ and so cannot be an inner automorphism of $D$. Suppose ${}^{(x,y)}(\Delta \psi)=\Delta(\psi\circ g^{-1}Z)$, for some $(x,y)\in D\times D$. Then, as in (\ref{algn:vert}), $c_y\circ\psi\circ c_{x^{-1}}=\psi\circ g^{-1}Z$. Therefore,
\begin{align}\label{algn:not_inner}
c_{\psi^{-1}(y)}=\psi^{-1}\circ c_y\circ\psi=g^{-1}Z\circ c_x,
\end{align}
a contradiction as $gZ$ is not inner, and so $S_\psi= I_{G\times G}({}_\psi(\cO D))$. Furthermore, ${}_\psi(\cO D)$ must extend to an $\cO S_\psi$-module since if it didn't, then
\begin{align*}
\rk_{\cO}(M)>\rk_{\cO}({}_\psi(\cO D)).[G\times G:S_\psi]=|D|.|L|=\rk_{\cO}(B),
\end{align*}
contradicting~\cite[Lemma 2.1]{li19}. Therefore $M\cong N\uparrow_{S_\psi}^{G\times G}$, where $N\downarrow^{S_\psi}_{D\times D}\cong {}_\psi(\cO D)$. Moreover, since $e_\varphi M e_\varphi=M$, $N\downarrow^{S_\psi}_{C\times C}\cong {}_\psi(\cO D)\otimes_{\cO}\cO Ze_\varphi$. Now
\begin{align}\label{algn:triv_D}
\begin{split}
({}_\psi(KD)\otimes_K KZe_\varphi)\uparrow_{C\times C}^{S_\psi}e_{1_D} \cong&({}_\psi(KD)e_{1_D}\otimes_K KZe_\varphi)\uparrow_{C\times C}^{S_\psi}\\
\cong&\bigoplus_{\chi\in\Irr(E_\psi|\varphi\otimes\varphi^{-1})}KS_\psi e_{\chi'},
\end{split}
\end{align}
where $\chi':=\Inf^{S_\psi}_{E_\psi}(\chi)$, the inflation of $\chi$ to $S_\psi$. Since $N$ is a direct summand of $({}_\psi(\cO D)\otimes_{\cO} \cO Ze_\varphi)\uparrow_{C\times C}^{S_\psi}$ and $\dim_K((K\otimes_{\cO}N)e_{1_D})=1$, we have that $(K\otimes_{\cO}N)e_{1_D}\cong KS_\psi e_{\chi'}$ as $KS_\psi$-modules, for some linear character $\chi\in \Irr(E_\psi|\varphi\otimes\varphi^{-1})$ and, for each such $\chi$, up to isomorphism, there is at most one extension $N_\chi$ of ${}_\psi(\cO D)\otimes_{\cO}\cO Ze_\varphi$ to an $\cO S_\psi$-module with $(K\otimes_{\cO}N_\chi)e_{1_D}\cong KS_\psi e_{\chi'}$. Until further notice we fix the appropriate linear $\chi\in \Irr(E_\psi|\varphi\otimes\varphi^{-1})$ such that $N\cong N_\chi$ as $\cO S_\psi$-modules.
\newline
\newline
Since $p\nmid|E_\psi|$, $\chi(g,h)\in\cO$, for all $(g,h)\in E_\psi$. We define $\alpha\in\Aut_{\cO}(B)$ via $\alpha(xe_\varphi)=\psi(x)e_\varphi$, for all $x\in D$ and $\alpha(ge_\varphi)=\chi(g,h)he_\varphi$, for all $(g,h)\in E_\psi$. One can readily check that $\alpha$ is a well-defined $\cO$-algebra automorphism of $B$. We demonstrate the most difficult condition:
\begin{align*}
&\alpha(ge_\varphi)\alpha(xe_\varphi)\alpha(g^{-1}e_\varphi)=h\psi(x)h^{-1}e_\varphi=hZ(\psi(x))e_\varphi=\psi_L(gZ)(\psi(x))e_\varphi\\
&=(\psi\circ gZ\circ \psi^{-1})(\psi(x))e_\varphi=\psi(gZ(x))e_\varphi=\alpha(gxg^{-1}e_\varphi),
\end{align*}
where we view $gZ,hZ\in\Aut(D)$ and the third equality follows from the fact that $(g,h)\in E_\psi$. We claim that $N\uparrow_{S_\psi}^{G\times G}\cong {}_\alpha B$.
\newline
\newline
Set $G_\psi:=\Delta\psi\rtimes E_\psi$ and $\cO_\alpha$ the $\cO G_\psi$-module $\cO$ affording $\Inf_{E_\psi}^{G_\psi}(\chi)$. Since $\langle e_\varphi\rangle_{\cO}\subseteq {}_\alpha B$ affords $\Inf_{E_\psi}^{G_\psi}(\chi)$, ${}_\alpha B\cong \cO_\alpha\uparrow_{G_\psi}^{G\times G}$. Next we consider the $\cO S_\psi$-module $\cO_\alpha\uparrow_{G_\psi}^{S_\psi}$. Now
\begin{align*}
\cO_\alpha\uparrow_{G_\psi}^{S_\psi}\downarrow^{S_\psi}_{D\times D}\cong \cO_\alpha\downarrow^{G_\psi}_{\Delta \psi}\uparrow_{\Delta \psi}^{D\times D}\cong \cO_{\Delta \psi}\uparrow_{\Delta\psi}^{D\times D}\cong {}_\psi(\cO D).
\end{align*}
Therefore, $\cO_\alpha\uparrow_{G_\psi}^{S_\psi}$ is an extension of ${}_\psi(\cO D)\otimes_{\cO}\cO Ze_\varphi$ to $\cO S_\psi$. In particular, $\dim_K((K\otimes_{\cO}\cO_\alpha\uparrow_{G_\psi}^{S_\psi})e_{1_D})=1$. Hence, since
\begin{align*}
\langle \Inf_{E_\psi}^{G_\psi}(\chi)\uparrow_{G_\psi}^{S_\psi},\Inf_{E_\psi}^{S_\psi}(\chi) \rangle_{S_\psi}&=\langle \Inf_{E_\psi}^{G_\psi}(\chi),\Inf_{E_\psi}^{S_\psi}(\chi)\downarrow^{S_\psi}_{G_\psi} \rangle_{G_\psi}\\
&=\langle \Inf_{E_\psi}^{G_\psi}(\chi),\Inf_{E_\psi}^{G_\psi}(\chi)\rangle_{G_\psi}=1,
\end{align*}
$(K\otimes_{\cO}\cO_\alpha\uparrow_{G_\psi}^{S_\psi})e_{1_D}$ affords $\Inf_{E_\psi}^{S_\psi}(\chi)$. Therefore, by the comments following (\ref{algn:triv_D}) and the fact that $(K\otimes_{\cO}N)e_{1_D}$ also affords $\Inf_{E_\psi}^{S_\psi}(\chi)$, $N\cong \cO_\alpha\uparrow_{G_\psi}^{S_\psi}$ and so $M\cong {}_\alpha B$.
\item What was proved in part (2) was essentially that any $M\in\cT(B)$ is uniquely determined by its vertex, say $\Delta\psi$, and the linear character $\chi\in\Irr(E_\psi|\varphi\otimes\varphi^{-1})$ afforded by $(K\otimes N)e_{1_D}$, where $N$ is an extension of ${}_\psi(\cO D)\otimes_{\cO}\cO Ze_\varphi$ to an $\cO S_\psi$-module and $M\cong N\uparrow_{S_\psi}^{G\times G}$. If $M\in\cT(B)$ induces the equivalence given by tensoring with $\lambda$, a linear character in $\Irr(G|1_C)$, then $\psi$ can be taken to be $\Id_D$ and $\alpha$ to be given by $\alpha(ge_\varphi)=\lambda(g)ge_\varphi$, for all $g\in G$. The corresponding $\chi$ is then given by $\chi(g,h)=\lambda(g)\varphi(gh^{-1})$, for all $(g,h)\in E_{\Id_D}=(Z\times Z).(\Delta E)$. Since $\chi$ is a linear character in $\Irr(E_{\Id_D}|\varphi\otimes\varphi^{-1})$, in fact every $\chi$ is of this form for some $\lambda\in\Irr(G|1_C)$. Therefore, every $M\in\cT(B)$, with vertex $\Delta D$, is given by tensoring with some linear character $\lambda\in\Irr(G|1_C)$.
\newline
\newline
Next suppose $M$ is given by tensoring with some linear character $\lambda\in\Irr(G|1_C)$. As shown above $M$ has vertex $\Delta D$. Let $N$ be the extension of $\cO D\otimes_{\cO}\cO Ze_\varphi$ to $\cO S_{\Id_D}$ from the proof of part (2). Since $S_{\Id_D}=I_{G\times G}(\cO D)$ and $M\cong N\uparrow_{S_{\Id_D}}^{G\times G}$, $N$ is in fact the unique summand of $M\downarrow^{G\times G}_{S_{\Id_D}}$ that extends $\cO D\otimes_{\cO}\cO Ze_\varphi$. Therefore, once we've fixed $\Delta D$ as a vertex, $\chi$ is uniquely determined by $M$. Finally, by the previous paragraph, $\lambda$ is uniquely determined by $\chi$.
\end{enumerate}
\end{proof}

\begin{thm}\label{thm:elem_ab}
Let $b$ be a block with normal abelian defect group and inertial quotient a product of elementary abelian groups. Then $\Picent(b)$ is trivial.
\end{thm}

\begin{proof}
By~\cite[Theorem A]{ku85} (see also~\cite[Theorem 6.14.1]{li18b} for a more detailed description), we may assume $b$ is of the form of $B$ as described just before Proposition~\ref{prop:triv_source}. Note that in~\cite[Theorem 6.14.1]{li18b}, the isomorphism classes of the defect group and inertial quotient do not change when we move from $b$ to $B$. In other words $D$ is abelian and $L=E/Z$ is a product of elementary abelian groups. Let $M\in \Picent(B)$. By~\cite[Propositions 4.3,4.4]{eali19}, $M\in\cT(B)$ and so we let $\psi$ and $\alpha$ be as in Proposition~\ref{prop:triv_source}.
\newline
\newline
We will use $\psi^*$ to denote the self-bijection of $\Irr(D)$ given by $\psi^*(\chi)(x)=\chi(\psi^{-1}(x))$, for all $x\in D$ and $\chi\in\Irr(D)$. For any $C\leq H\leq G$, we denote by ${}^\alpha H$ the unique subgroup $C\leq {}^\alpha H\leq G$ such that $\cO({}^\alpha H)e_\varphi=\alpha(\cO He_\varphi)$ or equivalently ${}^\alpha H/C=\psi_L(H/C)$. For any such $H$, we will use $\alpha^*$ to denote the bijection $\Irr(H|\varphi)\to\Irr({}^\alpha H|\varphi)$ given by $\alpha^*(\chi)(x)=\chi(\alpha^{-1}(x))$, for all $x\in {}^\alpha H$ and $\chi\in\Irr(H|\varphi)$. Since $M\in \Picent(B)$, when $H=G$, $\alpha^*$ is just the identity on $\Irr(B)$. Note that, since $\alpha$ permutes the left cosets of $C$ in $G$, $\alpha^*(\chi)\uparrow_{{}^\alpha H}^G=\alpha^*(\chi\uparrow_H^G)$, for any $C\leq H\leq G$ and $\chi\in\Irr(H|\varphi)$.
\newline
\newline
Decompose $D=D_1\times\dots\times D_n$ into indecomposable components with respect to the action of $L$. Let $\theta_i\in\Irr(D_i)\backslash\{1_{D_i}\}$, for each $1\leq i\leq n$ and set $\theta:=\theta_1\otimes\dots\otimes\theta_n\in\Irr(D)$. Since $\alpha^*(\theta\otimes\varphi)\uparrow_C^G=\alpha^*((\theta\otimes\varphi)\uparrow_C^G)=(\theta\otimes\varphi)\uparrow_C^G$, $\alpha^*(\theta\otimes\varphi)=\psi^*(\theta)\otimes\varphi$ must be conjugate to $\theta\otimes\varphi$ via an element of $E$. Therefore, by composing $\psi$ and $\alpha$ with an appropriately chosen $c_g$, we may assume that $\psi^*(\theta)=\theta$.
\newline
\newline
By an abuse of notation, we view each $\theta_i\in\Irr(D)$ by letting it act trivially on all the other $D_j$'s. An identical argument to the previous paragraph applied to each $\theta_i\otimes\varphi\in\Irr(C)$, for $1\leq i\leq n$, and the fact that $\psi^*$ already fixes $\theta$, gives that each $\theta_i$ is also fixed by $\psi^*$. In addition , by parts (3) and (4) of Lemma~\ref{lem:HonP}, $I_L(\theta_i)=C_L(D_i)$ and $I_E(\theta_i)=C_E(D_i)$.
\newline
\newline
Until further notice we fix a prime $\ell\mid|L|$. Since $L$ is a product of elementary abelian groups, $O_\ell(L)\cong (C_\ell)^t$, for some $t\in\mathbb{N}$. Therefore, by part (4) of Proposition~\ref{lem:HonP}, for each $1\leq i\leq n$,
\begin{align*}
O_\ell(L/I_L(\theta_i))=O_\ell(L/C_L(D_i))\cong \{1\}\text{ or }C_\ell.
\end{align*}
Also, since we're identifying $L$ with a subgroup of $\Aut(D)$,
\begin{align*}
\bigcap_{i=1}^n O_\ell(I_L(\theta_i))=\bigcap_{i=1}^n O_\ell(C_L(D_i))=O_\ell(C_L(D))=\{1\}.
\end{align*}
Therefore, viewing $O_\ell(L)$ as an $\mathbb{F}_\ell$-vector space, by Lemma~\ref{lem:vec_spc} there exist $1\leq i_1,\dots i_t\leq n$ such that $\bigcap_{j=1}^t O_\ell(I_L(\theta_{i_j}))=\{1\}$. In addition, setting
\begin{align*}
\vartheta_m:=\left(\bigotimes_{\substack{j=1\\j\neq m}}^t \theta_{i_j}\right)\otimes\left(\bigotimes_{\substack{j=1\\j\notin\{i_1,\dots, i_t\}}}^n 1_{D_j}\right)\otimes 1_{D_{i_m}}\in\Irr(D),
\end{align*}
for each $1\leq m\leq t$, Lemma~\ref{lem:vec_spc} also gives that
\begin{align*}
\cC_m:=O_\ell((I_L(\vartheta_m))=\bigcap_{\substack{j=1\\j\neq m}}^t O_\ell(I_L(\theta_{i_j}))\cong C_\ell,
\end{align*}
and that the $\cC_m$'s generate $O_\ell(L)$.
\newline
\newline
For any subgroup $H$ of $L$ we denote by $\tilde{H}$ its preimage in $E$. Until further notice we fix some $1\leq m\leq t$ and set $I_m:=I_L(\vartheta_m\otimes\varphi)$. Since $\psi^*(\theta_i)=\theta_i$, for all $1\leq i\leq n$, $\psi^*(\vartheta_m)=\vartheta_m$. Therefore, $\psi_L(I_m)=I_m$ and $\psi_L(\cC_m)=\cC_m$ implying ${}^\alpha(D\rtimes\tilde{I}_m)=D\rtimes\tilde{I}_m$ and ${}^\alpha(D\rtimes\tilde{\cC}_m)=D\rtimes\tilde{\cC}_m$.
\newline
\newline
By~\cite[Theorem 6.11(b)]{is76}, $\Irr(B|\vartheta_m\otimes\varphi)=\Irr(G|\vartheta_m\otimes\varphi)$ is in one-to-one correspondence, via induction, with $\Irr(I_G(\vartheta_m\otimes\varphi)|\vartheta_m\otimes\varphi)=\Irr(D\rtimes\tilde{I}_m|\vartheta_m\otimes\varphi)$. Therefore, since ${}^\alpha(D\rtimes\tilde{I}_m)=D\rtimes\tilde{I}_m$ and $\alpha^*$ respects induction, $\alpha^*$ must be the identity on $\Irr(D\rtimes\tilde{I}_m|\vartheta_m\otimes\varphi)$. By Lemma~\ref{lem:exten}, $\vartheta_m\otimes\varphi$ extends to $D\rtimes \tilde{\cC}_m$ and every extension to $D\rtimes \tilde{\cC}_m$ is stable in $D\rtimes \tilde{I}_m$. In other words, every character in $\Irr(D\rtimes\tilde{I}_m|\vartheta_m\otimes\varphi)$ lies above a unique character in $\Irr(D\rtimes \tilde{\cC}_m|\vartheta_m\otimes\varphi)$. Therefore, again since $\alpha^*$ respects induction, $\alpha^*$ is also the identity on $\Irr(D\rtimes \tilde{\cC}_m|\vartheta_m\otimes\varphi)$. Now every character in $\Irr(D\rtimes \tilde{\cC}_m|\vartheta_m\otimes\varphi)$ is linear and is therefore determined by its restriction to $\tilde{\cC}_m$. Hence, $\alpha$ is the identity on $Z(\cO\tilde{\cC}_m e_\varphi)=\cO\tilde{\cC}_m e_\varphi$.
\newline
\newline
Since the $\cC_m$'s generate $O_\ell(L)$ and the $O_\ell(L)$'s generate $L$ as $\ell$ runs over all primes dividing $|L|$, $\alpha$ is the identity on $\cO E e_\varphi$. In particular, $\psi_L$ is the identity. As noted earlier in this proof, for every $\chi\in\Irr(D)$, $\psi^*(\chi)$ is conjugate to $\chi$ via an element of $E$. Therefore, by part (5) of Lemma~\ref{lem:HonP}, $\psi^*$ is induced by an element of $E$ and so, by composing $\psi$ and $\alpha$ with an appropriately chosen $c_g$, we may assume that $\psi$ is the identity.
\newline
\newline
We now repeat the entire proof with $\psi$ being the identity until we've reproved that $\alpha$ is the identity on $\cO E e_\varphi$. Therefore, $\alpha$ is the identity and $M$ is the trivial Morita auto-equivalence.
\end{proof}

\begin{thm}\label{thm:cyc_prin}
Let $b$ be a block with normal abelian defect group and either cyclic inertial quotient or a principal block with abelian inertial quotient. Then $\Picent(b)$ is trivial.
\end{thm}

\begin{proof}
The proofs for both situations proceed similarly to that of Theorem~\ref{thm:elem_ab}. We first note that if $b$ is a principal block of a group $H$ with normal defect group $P$, then $P$ must be a Sylow $p$-subgroup of $H$. Therefore, by the Schur-Zassenhaus theorem, $H=P\rtimes F$, for some $p'$-subgroup $F\leq H$. Then, $C_F(P)$ must be in the kernel of $b$ and so, by factoring out by $C_F(P)$ and considering the relevant Morita equivalent block, we may assume that $b$ is already of the form of $B$ with $Z$ trivial. In other words, the block $B$ we reduce to is also principal. In particular, in both instances of the theorem, $E$ is abelian.
\newline
\newline
Now replace $\vartheta_m\in\Irr(D)$, in the proof of Theorem~\ref{thm:elem_ab}, with $1_D$. Proceeding as in that proof and noting that, since $E$ is abelian, every character in $\Irr(E|\varphi)$ is linear, we prove that $\alpha$ is the identity on $\cO Ee_\varphi$. We then continue to prove that $M$ is the trivial Morita auto-equivalence exactly as before.
\end{proof}

\section{Examples of non-trivial $\Picent$}

Picard groups have not been calculated for any Morita equivalence class of blocks with non-abelian defect group, excluding nilpotent blocks but, as our first example shows, just taking the principal $p$-block of a suitably chosen $p$-group already yields an example of a block with non-trivial $\Picent$.
\newline
\newline
We first set up some notation. A class-preserving automorphism of a group $H$ is an automorphism that leaves invariant every conjugacy class of $H$. Equivalently it leaves invariant every irreducible character of $H$. We define
\begin{align*}
\Aut_c(H)&=\{\alpha\in \textnormal{Aut}(H)| \alpha\textnormal{ is class-preserving}\},\\
\Inn(H)&=\{ c_h\in \textnormal{Aut}(H)| h\in G\},
\end{align*}
$\Out(H)=\Aut(H)/\Inn(H)$ and $\Out_c(H)=\Aut_c(H)/\Inn(H)$.

\begin{exam}\label{exam:OP}
Let $P$ be a finite $p$-group. We know, by~\cite{roggenkamp1987isomorphisms}, that $\Pic(\cO P)\cong \Hom(P,\cO^\times)\rtimes\Out(P)$. Since any element of $\Picent(\cO P)$ must fix the trivial character, $\Picent(\cO P)\cong \Out_c(P)$ and finding a non-inner, class-preserving automorphism of $P$ immediately yields a non-trivial element of $\Picent(\cO P)$. Many $p$-groups with such an automorphism have been constructed,~\cite{ya2013} lists some of these. The smallest $p$-group $P$ with $\Out_c(P)\neq \{1\}$ has order $32$ and is described in \cite{wall1947finite}.
\end{exam}

For blocks with abelian defect group the situation is different. As $\Out_c(P)$ is trivial, for any abelian $p$-group $P$, $p$-groups will no longer provide examples with non-trivial $\Picent$. In addition, in all the cases the Picard group of a block with abelian defect group has been computed (namely \cite{bkl18} and \cite{eali19}) $\Picent$ has been trivial. However, the families of blocks that we exhibit show explicitly that $\Picent$ of a block with abelian defect group is not trivial, in general. The first family we are going to describe yields a counterexample in the same spirit as Example~\ref{exam:OP}, the non-trivial element of $\Picent(B)$ is given by an outer automorphism of the relevant group.
\newline
\newline
Before stating the result we set up some notation. Let $t>1$ be an integer coprime to $p$. For any $i\in\mathbb{N}$, we denote by $\omega_i$ a primitive $i^{\nth}$ root of unity and take $n$ to be the smallest positive integer such that $\omega_{t^2}\in \mathbb{F}_{p^n}$. In other words, $n$ is the multiplicative order of $p\mod t^2$. Set $s:=(p^{nt}-1)/(t(p^n-1))$ and $D:=(C_p)^{nt}$ that we identify with $\mathbb{F}_{p^{nt}}$. Certainly $s$ is coprime to $p$ and
\[
(p^{nt}-1)/(p^n-1)=p^{n(t-1)}+...+p^n+1.
\]
Therefore, since each $p^{ni}\equiv 1\mod t^2$, $s\equiv 1 \mod t$. In particular, $s$ and $t$ are coprime. For any $\lambda\in\mathbb{F}_{p^{nt}}^\times$, we denote by $m_\lambda$ the group automorphism of $\mathbb{F}_{p^{nt}}$ given by multiplication by $\lambda$ and we denote by $\Phi_{p^n}$ the automorphism of $\mathbb{F}_{p^{nt}}$ given by $(x\mapsto x^{p^n})$.
\newline
\newline
We set $G:=D\rtimes E$, where $E:=(\langle g\rangle \rtimes \langle h\rangle)$, $\langle g\rangle=\langle m_{\omega_s}\rangle\cong C_s$ and $\langle h\rangle=\langle\Phi_{p^n}\circ m_{\omega_{t^2}}\rangle\cong C_{t^2}$. Since ${}^hg=g^{p^n}$, $E$ is well-defined. What's more, since $s$ and $t^2$ are coprime, $E/C_E(D)$ has order $st^2$. In other words, $E$ acts faithfully on $D$ and $\cO G$ is a block.

\begin{prop}\label{prop:st_example}
With the above notation, $\Out_c(G)$ is non-trivial. Moreover, the image of $\Out_c(G)$ in $\Picent(\cO G)$ is non-trivial.
\end{prop}

\begin{proof}
We define $\psi=m_{\omega_{t^2}}\in\Aut(D)$. This extends to an automorphism $\psi_G\in\Aut(G)$ that acts trivially on $E$. We claim that $\psi_G$ gives a non-trivial element of $\Out_c(G)$.
\newline
\newline
Note that $\psi\in\Aut(D)$ is not already in $E$. Suppose it were, then $\psi^{-1}\circ h=\Phi_{p^n}\in E$ is an element of order $t$. By considering their images in $E/\langle g\rangle\cong C_{t^2}$, we see that all the elements of order $t$ in $E$ are in $\langle g\rangle\times\langle h^t\rangle$ and hence powers of $h^t=m_{\omega_{t^2}}^t$. However, $\Phi_{p^n}$ is not of this form as it has non-trivial fixed points on $D$.
\newline
\newline
We now show that $\psi$ preserves the $E$-conjugacy classes of $D$, that is for each $x\in D$, there exists $y\in E$ such that $\psi(x)={}^y x$.
\newline
\newline
For $x\in \mathbb{F}_{p^{nt}}^\times$, $0\leq i<s$ and $0\leq j<t$,
\[
c_g^i\circ c_h^{1+jt}(x)=\psi(x)\Leftrightarrow \omega_s^{i}\omega_{t^2}^{1+jt}x^{p^{n(1+jt)}}=\omega_{t^2}x\Leftrightarrow x^{p^n-1}=\omega_{t^2}^{-jt}w_s^{-i}.
\]
Thus, since $x^{p^n-1}$ is a (not necessarily primitive) $(st)^{\nth}$ root of unity, there exist unique $i,j$ such that $c_g^i\circ c_h^{1+jt}(x)=\psi(x)$.
\newline
\newline
Next we prove that for all $x\in D\backslash\{1\}$, $C_E(x)=\{1\}$. This is crucial, since it will give us a description of the irreducible characters of the group $G$.
\newline
\newline
We first note that all the non-trivial elements of $\langle g\rangle\times \langle h^t\rangle=\langle m_{\omega_{st}}\rangle$ have no non-trivial fixed points on $D$. If $y\in E\backslash \langle g\rangle$, then, by considering its image in $E/\langle g\rangle\cong C_{t^2}$, there exists an integer $l$ such that $y^l\in(\langle g\rangle\times \langle h^t\rangle)\backslash\{1\}$. If $y$ has non-trivial fixed points on $D$ then $y^l$ would have non-trivial fixed points as well, thus $y$ can't have any non-trivial fixed points. In other words, for all $x\in D\backslash\{1\}$, $C_E(x)=\{1\}$.
\newline
\newline
We now show that the analogous property holds, when we replace $D$ with $\Irr(D)$. Let $y\in E\backslash\{1\}$ and decompose $D=D_1\times\dots\times D_m$ with respect to the action of $\langle y \rangle$. Let $\theta_i\in\Irr(D_i)$, for each $1\leq i\leq m$, and set $\theta:=\theta_1\otimes\dots\otimes\theta_m\in\Irr(D)$. If $\theta\neq 1_D$, say $\theta_j\neq 1_{D_j}$, for some $1\leq j\leq m$, then
\begin{align*}
I_{\langle y\rangle}(\theta)\subseteq I_{\langle y\rangle}(\theta_j)=C_{\langle y\rangle}(D_j)=\{1\},
\end{align*}
where the first equality follows from parts (3) and (4) of Lemma~\ref{lem:HonP} and the final one from the previous paragraph. In other words, the action of $E$ on $\Irr(D)$ is fixed-point-free as well. Thus all irreducible character of $G$ are either irreducible characters with $D$ in their kernel or irreducible characters induced from non-trivial characters of $D$. Characters of $G$ with $D$ in their kernel correspond to characters of $G/D$, that are obviously fixed by $\psi_G$, while the characters of $G$ induced from $D$ are characterised by their values on $G$-conjugacy classes of elements of $D$ but $\psi$ preserves the $E$-conjugacy classes of elements of $D$. Therefore, every character of $G$ is $\psi_G$-stable and so $\psi_G\in \Aut_c(G)$.
\newline
\newline
Note that $\psi_G$ is a non-inner group automorphism of $G$, since we have shown that $\psi$ is not induced by an element of $E$. Finally, we note that, since ${}_{\psi_G}(\cO G)\cong\cO_{\Delta\psi_G}\uparrow_{\Delta \psi_G}^{G\times G}$, ${}_{\psi_G}(\cO G)$ has vertex $\Delta\psi$. Similarly $\cO G$ has vertex $\Delta D$. Since $\psi$ is not induced by an element of $E$, $\Delta\psi$ is not conjugate to $\Delta D$ in $G$. Therefore, ${}_{\psi_G}(\cO G)$ is a non-trivial element of $\Picent(\cO G)$.
\end{proof}

We now exhibit another family of blocks with non-trivial $\Picent$. The main difference with the previous family is that, as well as having abelian defect group, these blocks also have abelian inertial quotient and diagonal vertex. In fact, the non-trivial element of $\Picent(B)$ we exhibit is given by multiplication by a linear character of the inertial quotient. Note that, by Theorem~\ref{thm:cyc_prin}, these blocks are necessarily non-principal.
\newline
\newline
Let $\ell$ be a prime number different from $p$. Take
\[
E:=(\langle z\rangle\times \langle g\rangle)\rtimes\langle h\rangle\cong(C_{\ell}\times  C_{\ell^2})\rtimes C_{\ell},
\]
defined by $hz=zh$, ${}^hg=gz$ and set $F:=\langle z\rangle\times \langle g^{\ell}\rangle\times \langle h\rangle\leq E$. Furthermore, let $n$ be the multiplicative order of $p\mod \ell^2$ and set $D:=(C_p)^n\times(C_p)^n=:D_1\times D_2$ that we identify with $\mathbb{F}_{p^n}\oplus \mathbb{F}_{p^n}$.
\newline
\newline
We define $G:=D\rtimes E$, where the action of $E$ on $D_1$ has kernel $\langle z\rangle \times \langle h\rangle$, while on $D_2$ the kernel is $\langle z\rangle\times \langle g^{\ell}h\rangle$ and $g$ acts on both components by multiplication by $\omega\in \mathbb{F}_{p^n}^{\times}$, a primitive $\ell^2$-th root of of unity. Since $\omega\notin\mathbb{F}_{p^m}^{\times}$ for any $m<n$, no proper $\mathbb{F}_p$-subspace of $\mathbb{F}_{p^n}$ has an $\mathbb{F}_p$-linear map with eigenvalue $\omega$. In particular, the actions of $E$ on both $D_1$ and $D_2$ are indecomposable. Of course, $C_E(D)=Z:=\langle z\rangle$. Let $\varphi\in\Irr(Z)\backslash\{1_Z\}$ and set $B:=\mathcal{O}Ge_{\varphi}$.

\begin{prop}\label{prop:tensor_linear}
With the notation above, tensoring by any non-trivial $\lambda\in\Irr(G|1_{D\rtimes F})$ yields a non-trivial element of $\Picent(B)$.
\end{prop}

\begin{proof}
Let $\lambda\in\Irr(G|1_{D\rtimes F})$. Every character in $\Irr(B| 1_D)$ is of the form $\Inf_E^G(\chi)$, for some $\chi\in\Irr(E|\varphi)$. Then, by~\cite[Lemma 3.1(2)]{li2019}, we have $\chi^{\oplus m}=\eta\uparrow_{Z(E)}^E$, for some $m\in\mathbb{N}$ and $\eta\in \Irr(Z(E)| \varphi)$. Note that $Z(E)=Z\times \langle g^\ell\rangle\leq F$. Therefore,
\begin{align*}
(\lambda\downarrow^G_E).(\eta\uparrow_{Z(E)}^E)=(\lambda\downarrow^G_{Z(E)}.\eta)\uparrow_{Z(E)}^E=\eta\uparrow_{Z(E)}^E,
\end{align*}
giving that $(\lambda\downarrow^G_E).\chi=\chi$ and finally that $\lambda.\Inf_E^G(\chi)=\Inf_E^G(\chi)$.
\newline
\newline
Now take $\chi\in\Irr(B|\theta)$, for some $\theta=\theta_1\otimes\theta_2\in \Irr(D)$, where $\theta_i\in \Irr(D_i)\backslash\{1_{D_i}\}$, for $i=1,2$. By part (3) and (4) of Lemma~\ref{lem:HonP}, $I_E(\theta_i)=C_E(D_i)$, for $i=1,2$. Therefore, $I_E(\theta_1\otimes\theta_2)=C_E(D_1)\cap C_E(D_2)=Z$ and so $\chi=(\theta\otimes\varphi)\uparrow_{D\times Z}^G$. Hence,
\begin{align*}
\lambda.\chi=\lambda.((\theta\otimes\varphi)\uparrow_{D\times Z}^G)=((\lambda\downarrow^G_{D\times Z}).(\theta\otimes\varphi))\uparrow_{D\times Z}^G=(\theta\otimes\varphi)\uparrow_{D\times Z}^G=\chi.
\end{align*}
Finally, let $\chi\in\Irr(B|\theta)$, where $\theta=\theta_1\otimes 1_{D_2}\in \Irr(D)$, for some $\theta_1\in \textnormal{Irr}(D_1)\backslash\{1_{D_1}\}$. This time $I_E(\theta)=\langle z\rangle\times\langle h\rangle\leq F$ and the argument concludes similarly to the previous paragraph. Of course, the case of $\theta=1_{D_1}\otimes \theta_2$, for some $\theta_2\in \textnormal{Irr}(D_2)\backslash\{1_{D_2}\}$ is dealt with in an identical fashion.
\newline
\newline
We have proved that tensoring with $\lambda$ fixes $\Irr(B)$ pointwise and thus defines an element of $\Picent(B)$. The fact that non-trivial $\lambda$ induce non-trivial elements of $\Picent(B)$ follows from part (3) of Proposition~\ref{prop:triv_source}.
\end{proof}

\end{document}